\documentclass[10pt]{article}
\usepackage{amsmath, amsfonts, amsthm, amssymb, hyperref, setspace, geometry, authblk, color}

\title{\bf On the group pseudo-algebra of finite groups}

\author[1]{Mark L. Lewis}
\author[1]{Quanfu Yan}

\affil[1]{Department of Mathematical Sciences, Kent State University, Kent, OH 44242, USA}

{
    \makeatletter
    \renewcommand\AB@affilsepx{: \protect\Affilfont}
    \makeatother

    \affil[ ]{\par Email addresses}
    
    \makeatletter
    \renewcommand\AB@affilsepx{, \protect\Affilfont}
    \makeatother

    \affil[ ]{\href{mailto:lewis@math.kent.edu}{lewis@math.kent.edu}}
    \affil[ ]{\href{mailto:qyan5@kent.edu}{qyan5@kent.edu}}

}

\renewcommand{\Affilfont}{\small\it}
\date{}
\usepackage{abstract}

\newtheorem{theorem}{Theorem}[section]

\newtheorem{lemma}[theorem]{Lemma}

\theoremstyle{definition}

\newtheorem*{conjecture*}{Conjecture A}
\allowdisplaybreaks

\expandafter\let\expandafter\oldproof\csname\string\proof\endcsname
\let\oldendproof\endproof
\renewenvironment{proof}[1][\proofname]{%
  \oldproof[\bfseries\scshape #1]%
}{\oldendproof}

\usepackage{stackengine,scalerel}
\stackMath
\def\trianglelefteqslant{\ThisStyle{\mathrel{%
  \stackinset{r}{.75pt+.15\LMpt}{t}{.1\LMpt}{\rule{.3pt}{1.1\LMex+.2ex}}{\SavedStyle\leqslant}%
}}}
\renewcommand{\unlhd}{\trianglelefteqslant}

\renewcommand{\leq}{\leqslant}
\renewcommand{\geq}{\geqslant}
\begin{document}
\maketitle
\begin{abstract}
\noindent\textbf{Abstract.} Let $G$ be a finite group. The group pseudo-algebra of $G$ is defined as the multi-set $C(G)=\{(d,m_G(d))\mid d\in{\rm Cod}(G)\},$ where $m_G(d)$ is the number of irreducible characters of with codegree $d\in {\rm Cod}(G)$. We show that there exist two finite $p$-groups with distinct orders that have the same group pseudo-algebra, providing an answer to Question 3.2 in \cite{Moreto2023}. In addition, we also discuss under what hypothesis two $p$-groups with the same group pseudo-algebra will be isomorphic.

\medskip

\end{abstract}

\section{Introduction}

All groups considered in this article are finite. As usual, $G$ will always be a finite group, and $k(G)$ denotes the number of conjugacy classes of $G.$ We write ${\rm Irr}(G)$ to denote the set of complex irreducible characters of $G$ and ${\rm cd}(G)=\{\chi(1)\mid \chi\in {\rm Irr}(G)\}.$ Let
$\chi\in {\rm Irr}(G)$. The codegree of $\chi$ is defined as $${\rm cod}\chi=\frac{|G:\ker\chi|}{\chi(1)},$$ which was introduced by Qian, Wang and Wei in \cite{Qian1}. The concept has been studied extensively and proved to have interesting connections with some algebraic structure of finite groups (see, for example, \cite{Moreto2023,Isaacs2011,Qian2021,Yang2017,Du2016}).

In \cite{Moreto2023}, A. Moret\'{o} first introduced the concept of the group pseudo-algebra, which is defined as the multi-set $$C(G)=\{(d,m_G(d))\mid d\in{\rm Cod}(G)\},$$ where ${\rm Cod}(G)=\{{\rm cod}\chi\mid  \chi\in {\rm Irr}(G)\}$ and $m_G(d)$ is the number of irreducible characters having codegree $d$. He showed that if two finite abelian groups have the same group pseudo-algebra, then they are isomorphic. Additionally, a natural question arises: must groups have the same order if they have the same group pseudo-algebra? A particular case of the question asks whether $G\cong A$ provided that $C(G)=C(A),$ where $G$ is a finite group and $A$ is an abelian $p$-group for some prime $p.$ He gave an affirmative answer when either $A$ is cyclic or the exponent of $A$ does not exceed $p^2$ {\rm (see \cite[Theorem 3.4]{Moreto2023})}. Thus, the next natural case to look is when the exponent of $A$ is $p^3.$ We will prove that if $G$ is a nonabelian group and $A$ is an abelian $p$-group of exponent $p^3$ so that $C(G)=C(A)$, then $p=2.$ This result guides us in constructing examples, suggesting that the question does not always yield a positive answer.

\begin{theorem}\label{thm1.1} Let $p$ be a prime. There exists an abelian $p$-group $A$ and a group $G$ with $C(G)=C(A)$ so that $G$ is not isomorphic to $A.$ Hence, groups may have different orders even though they have the same group pseudo-algebra.
\end{theorem}

From the above theorem, it is clear that the question is not always true when the abelian $p$-group $A$ has three generators. Thus, it might be a good idea to focus on the case when $A$ has two generators, in particular, when $A \cong C_{p^n}\times C_p,$ where $p$ is a prime.  We first consider metacyclic groups $G$ satisfying $C(G)=C(A)$ and it turns out that $G$ must be isomorphic to $A.$ Applying this result, we can prove that for a group $G$ and $A \cong C_{p^n}\times C_p,$ if $C(G)=C(A),$ then either $G\cong A$ or $|G:G'|=p^2$ and $p>2.$  

\begin{theorem}\label{thm1.2} Let $G$ be a group and $A\cong C_{p^n}\times C_p,$ where $p$ is a prime and $n\geq 3$ is an integer. Suppose that $C(G)=C(A).$ Then $G$ is a $p$-group and one of the following holds:

$(1)$ $G\cong A.$

$(2)$ $|G:G'|=p^2,$ $p>2$ and $Z(G)$ is noncyclic. In addition, there is a unique maximal subgroup $X$ of $G'$ which is normal in $G$ so that the factor group $G/X$ is nonabelian of order $p^3$ and of exponent $p.$
\end{theorem}

Applying the above result, we show that $G\cong A$ if $A\cong C_{p^3}\times C_p$ and $C(G)=C(A).$ We also demonstrate that under the same hypothesis as stated in the above theorem, if, in addition, $G$ has either a metacyclic maximal subgroup or a two-generator derived subgroup $G',$ then $G\cong A.$  

\begin{theorem}\label{thm1.3} Let $G$ be a group and $A\cong C_{p^n}\times C_p,$ where $p$ is a prime. Suppose that $C(G)=C(A).$ Then $G\cong A$ if one of the following holds:

$(1)$ $G$ has a metacyclic maximal subgroup,

$(2)$ The derived subgroup $G'$ is generated by two elements,

$(3)$ The derived subgroup $G'$ is abelian.\\
 \end{theorem} 
The work in this paper was completed by the second author (P.h.D student) under the supervision of the first author at ***. The contents of this paper may appear as part of the second author's P.h.D dissertation.

\section{Main Results}
In this section, we start by stating a fact that will be used frequently. Theorem A in \cite{Qian1} yields that if $G$ is a group such that ${\rm Cod}(G)$ is a set of powers of a prime $p,$  then $G$ is a $p$-group. Now we prove the following basic lemmas.

\begin{lemma}\label{lem2.1}Let $G$ be a nonabelian group and $A$ be an abelian $p$-group of order $p^a$ for some prime $p$. Suppose that $C(G)=C(A).$ Then $|{\rm cd}(G)|\geq 3.$
\end{lemma}

\begin{proof} Since $A$ is abelian, we have that $A\cong {\rm Irr}(A)$ and so ${\rm Cod}(A)$ coincides with the set of element orders of $A$. Hence ${\rm Cod}(G)={\rm Cod}(A)$ is a set of powers of $p.$ It follows that $G$ is a $p$-group. 
Assume that $|{\rm cd}(G)|<3.$ Then $|{\rm cd}(G)|=2$ as $G$ is nonabelian. So we may assume ${\rm cd}(G)=\{1, p^e\}$ for some positive integer $e.$ Let $|A|=p^a$, $|G|=p^n$ and $|G:G'|=p^r.$ By $C(G)=C(A),$ we have that $k(G)=k(A)=p^a.$ Then $p^r<p^a<p^n.$ Notice that $|G|=|G:G'|+(k(G)-|G:G'|)p^{2e}.$ Thus $p^{n-r}-1=(p^{a-r}-1)p^{2e},$ contrary to the fact that $p^{2e}$ does not divide $p^{n-r}-1.$ Hence $|{\rm cd}(G)|\geq 3,$ as wanted.
\end{proof}

Now we delve deeper into the degree set of $G.$ In particular, we consider the case when ${\rm cd}(G)=\{1,p,p^2\}.$ 

\begin{lemma}\label{lem2.2} 
 Let $G$ be a nonabelian group and $A$ be an abelian $p$-group of order $p^a$ for some prime $p$. Suppose that $C(G)=C(A)$ and ${\rm cd}(G)=\{1,p,p^2\}.$ Then $p=2$ and $|G|=2^{a+2}.$ In particular, if we write $|G:G'|=p^r$ and $k_1=|\{\chi\in {\rm Irr}(G)\mid \chi(1)=p\}|$ and $k_2=|\{\chi\in {\rm Irr}(G)\mid \chi(1)=p^2\}|,$ then $k_1=p^a-p^r-p^{r-2}$ and $k_2=p^{r-2}.$ 
\end{lemma}

\begin{proof} Following the same reasoning process as in Lemma \ref{lem2.1}, $G$ is a $p$-group. Since $C(G)=C(A)$, we have that $k(G)=k(A)=p^a$ and so $|G|>p^a$ as $G$ is nonabelian. 
On the other hand, we have that
\begin{eqnarray*}
|G|&=&|G:G'|+k_1p^{2}+k_2p^4\\
&\leq& p^r+p^2+(p^a-p^r-1)p^4\\
&=& p^{a+4}-p^{r+4}+p^r-p^4+p^2\\
&<&p^{a+4}.
\end{eqnarray*}
Hence, $|G|=p^{a+1},p^{a+2},\ {\rm or}\ p^{a+3}.$ Notice that $|G:G'|=k(G)-k_1-k_2$ and so $|G|=k(G)+k_1(p^2-1)+k_2(p^4-1).$ Hence $(p^2-1)\mid (|G|-k(G))$, which indicates that $|G|=p^{a+2}.$  By \cite[Theorem 3]{GMMP1998}, such groups do not exist if $p$ is odd. Hence, $p=2.$ Now it is easy to see that $k_1=p^a-p^r-p^{r-2}$ and $k_2=p^{r-2}$. 
\end{proof}

With the above lemmas, we are prepared to provide an example for Theorem \ref{thm1.1}, and it is advisable to look at $2$-groups. Let $A\cong C_{2^3}\times C_2\times C_2$ be an abelian $2$-group. Then $C(A)=\{(1,1), (2,7), (2^2,8), (2^3,16)\}.$ Since $\chi(1)<{\rm cod}\chi$ for all non-principal characters $\chi$ of $G,$ it follows that ${\rm cd}(G)$ is a subset of $\{1,2,2^2\}.$  By Lemma \ref{lem2.1} and Lemma \ref{lem2.2}, if there is a group $G$ so that  $C(G)=C(A)$, then $|G|=2^7$ and ${\rm cd}(G)=\{1,2,4\}.$ We notice that such a group does exist. For example, using GAP, $G$ can be one of {\rm SmallGroup}$(128,755)$, {\rm SmallGroup}$(128,756),$ {\rm SmallGroup}$(128,773)$. 

Since there is a counterexample when $A$ has three generators, we move on to the case when $A=C_{p^n}\times C_p$, where $p$ is a prime. If $G$ is a group satisfying that $C(G)=C(A)$, then by \cite[Lemma 3.3]{Moreto2023}, $G/\Phi(G)\cong A/\Phi(A)\cong C_p\times C_p,$ and hence $G$ also has two generators. We now consider a class of two-generator groups: metacyclic groups. We give the following result.

\begin{theorem}\label{thm2.3}
Let $G$ be a metacyclic group and  $A=C_{p^n}\times C_p$, where $p$ is a prime. Suppose that $C(G)=C(A).$ Then   $G$ is abelian and $G\cong A.$
\end{theorem}  

\begin{proof}[Proof of Theorem \ref{thm2.3}] Assume that $G$ is nonabelian. Clearly $G$ is a $p$-group and $k(G)=k(A)=p^{n+1}.$ Since $p^2-1 \mid (|G|-k(G))$, we have that $|G|\geq p^{n+3}.$ If $p$ is odd, then it follows from  \cite[Corollary 2.3]{HK2011} that $k(G)$ is not a power of $p,$ a contradiction. Hence $p=2.$ 
Since $G/\Phi(G)\cong A/\Phi(A)\cong C_2 \times C_2$ and $C(G)=C(A),$ $G/G'$ must be isomorphic to a subgroup of $A$. So we may assume that $G/G'\cong C_{2^{a-1}}\times C_2,$ where $2\leq a\leq n.$ Since $C(G)=C(A)=\{(1,1),(2,3),(2^2,2^2),...,(2^n,2^n)\}$ and $C(G/G')=\{(1,1),(2,3),(2^2,2^2),...,(2^{a-1},2^{a-1})\},$ we have that ${\rm cod}\chi\geq 2^a$ for any nonlinear irreducible $\chi$ of $G$. 

We first claim that $|G'|\geq 2^2$ and $a\geq 3.$ Let $|G|=2^{m}.$ If $|G'|=2,$ then $G/G'\cong C_{2^{m-2}}\times C_2$ and so there is an irreducible character $\chi$ of $G$ such that ${\rm cod}\chi=2^{m-2}.$ Notice that $m\geq n+3$ and so $2^{m-2}\geq 2^{n+1}.$ This is impossible because the largest codegree of $G$ is $2^n.$ Now assume that $|G:G'|=4.$ Notice that nonabelian $2$-groups with $|G:G'|=4$ have been classified {\rm (see \cite[Proposition 1.6]{Berkovich2008})}. In particular, the number of conjugacy classes of such groups is not a power of $2.$   Hence $|G:G'|>4$ and so the claim holds. 

Notice that \cite[Theorem 4.6]{King1973} implies that $G$ has a unique minimal normal subgroup, say $S$, in its derived group such that $G/S$ splits. Now we consider the factor group $\overline{G}=G/S.$ Thus we may assume that $\overline{G}=\overline{N}\cdot\overline{K}$ with $\overline{N}\cap\overline{K}=1$, where $\overline{N}\unlhd \overline{G}$ and $\overline{K}\leq \overline{G}.$  
Clearly, $\overline{G}'\leq \overline{N}$ and $\overline{G}/\overline{G'}\cong G/G'.$ Hence by the above claim we have that $|K|>2$ and $|\overline{N}:\overline{G}'|=2.$ 

 Based on a Blackburn's result \cite[Lemma 2.2]{Black1958}, we let $\overline{X}$ be the unique maximal subgroup of $\overline{G}'$ which is normal in $\overline{G}.$ Consider the factor group $\Tilde{G}=\overline{G}/\overline{X}=\overline{N}/\overline{X}\cdot \overline{K}\overline{X}/\overline{X}=\Tilde{N}\cdot\Tilde{K}.$ Clearly, $|\Tilde{G}|=2^{a+1},$ $|\Tilde{G}'|=2$ and $\Tilde{G}/\Tilde{G}'\cong G/G'\cong C_{2^{a-1}}\times C_2.$ Since for any irreducible character $\chi\in {\rm Irr}(\Tilde{G})$ with $\chi(1)>1$, ${\rm cod}\chi=\frac{2^{a+1}}{\chi(1)|{\rm ker}\chi|}\geq 2^a.$ Then $\chi(1)\leq 2$ and ${\rm ker}\chi=1.$ This implies that ${\rm cd}(\Tilde{G})=\{1,2\}$ and all nonlinear irreducible characters of $\Tilde{G}$ are faithful. Hence $\Tilde{G}'$ is the unique minimal normal subgroup of $\Tilde{G}.$ Since $\Tilde{G}$ is a 2-group, we have that $\Phi(\Tilde{G})=\langle x^2\mid x\in \Tilde{G}\rangle.$ Let $x,y\in \Tilde{G},$ then by the fact $|\Tilde{G}'|=2$ we have that $1=[x,y]^2=[x,y^2]$ and so $y^2\in Z(\Tilde{G}).$ Hence $\Phi(\Tilde{G})\leq Z(\Tilde{G}).$ As $\Tilde{G}/\Phi(\Tilde{G})\cong C_2\times C_2$ and $\Tilde{G}$ is not abelian, it follows that $\Phi(\Tilde{G})=Z(\Tilde{G}).$ If $\Tilde{K}\cap Z(\Tilde{G})>1,$ then by the uniqueness of $\Tilde{G}'$, $\Tilde{G}'\leq \Tilde{K}\cap Z(\Tilde{G})\leq \Tilde{K},$ contrary to the fact $\Tilde{K}\cap \Tilde{G}'\leq \Tilde{K}\cap \Tilde{N} =1.$ Hence $\Tilde{K}\cap Z(\Tilde{G})=1.$ Since $|\Tilde{G}:\Tilde{K}Z(\Tilde{G})|=(|\Tilde{G}|/|Z(\Tilde{G})|)/|\Tilde{K}|\leq 1$ as $|\Tilde{K}|=|\overline{K}|\leq 2^2,$ $\Tilde{G}=\Tilde{K}Z(\Tilde{G})$ and so $\Tilde{G}/Z(\Tilde{G})\cong \Tilde{K}$ is cyclic. It follows that $\Tilde{G}$ is abelian. This is a contradiction. 

Therefore, $G$ is abelian. Then $G\cong {\rm Irr}(G)$ and so $C(G)$ determines the number of elements of each order. It follows that $G\cong A.$ Now the proof is complete.
\end{proof}

Next, we give a proof of Theorem \ref{thm1.2}. 

\begin{proof}[Proof of Theorem \ref{thm1.2}] If $G$ is abelian, then $(1)$ follows. Assume now that $G$ is nonabelian and $G/G'\cong C_{p^{a-1}}\times C_p,$ where $2\leq a\leq n.$  A proof similar to Theorem \ref{thm2.3} shows that $G$ is a $p$-group, $k(G)=k(A)=p^{n+1}$ and $|G|\geq p^{n+3}.$
Since $G$ is a nonabelian $p$-group with two generators, by \cite[Lemma 2.2]{Black1958} we can let $X$ be the unique maximal subgroup of $G'$ which is normal in $G.$ Consider the factor group $\overline{G}=G/X.$ Then $|\overline{G}|=p^{a+1}$, $\overline{G}/\overline{G}'\cong G/G',$ and $\overline{G}'=G'/X$ has order $p.$  Notice that for any irreducible character $\chi\in {\rm Irr}(\overline{G})$ with $\chi(1)>1$, ${\rm cod}\chi=\frac{p^{a+1}}{\chi(1)|{\rm ker}\chi|}\geq p^a.$ Then $\chi(1)\leq p$ and ${\rm ker}\chi=1.$ This implies that ${\rm cd}(\overline{G})=\{1,p\}$ and all nonlinear irreducible characters are faithful. Hence $\overline{G}'$ is the unique minimal normal subgroup of $\overline{G}.$  

Write $\overline{G}/\overline{G}'=\overline{C}/\overline{G}'\times \overline{D}/\overline{G}',$ where $\overline{C}/\overline{G}'\cong C_{p^{a-1}}$ and $\overline{D}/\overline{G}'\cong C_p.$ Next, we will discuss in two cases.

{\bf Case 1: $|G:G'|>p^2.$} 

We claim that that $G$ is metacylic.   By \cite[Theorem 2.3]{Black1958}, we only need to show that $\overline{G}$ is metacyclic. If $\Phi(\overline{C})=1,$ then $\overline{C}$ is elementary abelian and so $\overline{C}/\overline{G}'\cong C_p$ and $a=2$, contrary to $|G:G'|>p^2.$  As $\Phi(\overline{C})$ char $\overline{C} \unlhd \overline{G}$, it follows that $\Phi(\overline{C}) \unlhd \overline{G}.$ Hence by the uniqueness of $\overline{G}'$, we have that $\overline{G}'\leq \Phi(\overline{C}).$ Notice that $(\overline{C}/\overline{G}')/(\Phi(\overline{C})/\overline{G}')\cong \overline{C}/\Phi(\overline{C})$ is cyclic. Then $\overline{C}$ is cyclic. Since $\overline{G}/\overline{C}$ is cyclic, $\overline{G}$ is metacylic. Hence, the above claim holds. It follows from Theorem \ref{thm2.3} that $G$ is abelian, which is a contradiction. Hence, this case cannot happen.
 
{\bf Case 2: $|G:G'|=p^2.$} 

 If $p=2,$ then $|G:G'|=4$ and such groups have been classified. By the equation $|G|=|G:G'|+(k(G)-|G:G'|)\cdot 2^2$, it is easy to see that $k(G)$ is not a power of $2,$ contrary to the hypothesis $C(G)=C(A).$ Hence $p>2.$ Since $\overline{G}/\overline{G}'\cong G/G'\cong C_p\times C_p,$ it follows that $\overline{G}$ has order $p^3.$ Now we only need to show that $\overline{G}$ has exponent $p.$ If $\overline{G}$ has exponent $p^2$, then the group $\overline{C}$ defined above is cyclic of order $p^2$ and hence $\overline{G}$ is metacyclic and so $G$ must be metacylic. It follows from Theorem \ref{thm2.3} that $G$ is abelian, which is a contradiction. Hence $\overline{G}$ is of exponent $p$ and so (2) follows. 
\end{proof}

By Lemma \ref{lem2.1} and Lemma \ref{lem2.2}, we have that if $A\cong C_{p^3}\times C_p$ and $G$ is a group satisfying $C(G)=C(A),$ then $G\cong A.$ Theorem \ref{thm1.3} immediately follows from Theorem \ref{thm1.2} and a result of Blackburn. He proved that if a $p$-group $G$ and its derived subgroup $G'$ are generated by two elements, then $G'$ is abelian (see \cite[Theorem 4]{Black1957}).   

\begin{proof}[Proof of Theorem \ref{thm1.3}] If $G$ is abelian, there is nothing to prove. If $G$ is not isomorphic to $A$, then by Theorem \ref{thm1.2}, we have that $|G:G'|=p^2$ and $p>2.$ If $G$ has a metacyclic maximal subgroup $M$, then $G'\leq \Phi(G)$ is a subgroup of $M$ and so $G'$ is metacyclic, which indicates that $G'$ is generated by two elements. Hence by Blackburn's result, $G'$ is abelian. Now we have that $G'$ is abelian in all three cases. Notice that $\chi(1)\mid |G:G'|$ for all irreducible characters $\chi$ of $G.$ Then ${\rm cd}(G)$ is a subset of $\{1,p,p^2\}.$ It follows from Lemma \ref{lem2.1} and \ref{lem2.2} that $p=2.$ This is a contradiction. Therefore, $G\cong A,$ as desired. 
\end{proof}

\end{document}